\documentclass[11pt]{amsproc}

\usepackage{enumerate}

\usepackage{xargs}
\usepackage{amsmath,amssymb,amsfonts,amsthm}
\usepackage{mathtools}

\usepackage[backrefs]{amsrefs}

\usepackage{etoolbox}
\AtBeginEnvironment{biblist}{\catcode`\#=12 }

\newcommand{\abs}[1]{\lvert#1\rvert}
\newcommand{\lrabs}[1]{\left\lvert#1\right\rvert}
\DeclareMathOperator{\id}{id}
\renewcommand{\Im}{{\rm Im}}
\renewcommand{\Re}{{\rm Re}}
\DeclareMathOperator{\SL}{SL}
\DeclareMathOperator{\GL}{GL}
\newcommand{\defeq}{\coloneqq}
\newcommand{\nats}{\mathbb N}
\newcommand{\ints}{\mathbb Z}
\newcommand{\rats}{\mathbb Q}
\newcommand{\reals}{\mathbb R}
\newcommand{\complex}{\mathbb C}
\newcommand{\finfield}{\mathbb F}
\DeclareMathOperator{\Sym}{Sym}
\newcommandx{\set}[2][2=\empty]{\{#1\ifx#2\empty\else\mid#2\fi\}}
\newcommandx{\gensubalg}[2][2=\empty]{{\langle#1\ifx#2\empty\else\mid#2\fi\rangle}}
\newcommandx{\gensubgrp}[2][2=\empty]{\langle#1\ifx#2\empty\else\mid#2\fi\rangle}
\newcommand{\card}[1]{\lvert#1\rvert}
\newcommandx{\gensubsp}[2][2=\empty]{\langle#1\ifx#2\empty\else\mid#2\fi\rangle}
\newcommand{\Prob}{\mathbf{P}}
\DeclareMathOperator{\SO}{SO}
\DeclareMathOperator{\U}{U}
\renewcommand{\O}{{\rm O}}
\newcommandx{\lrset}[2][2=\empty]{\left\{#1\ifx#2\empty\else\mid#2\fi\right\}}
\DeclareMathOperator{\diag}{diag}
\DeclareMathOperator{\Cay}{Cay}
\newcommand{\trivgrp}{\mathbf{1}}

\theoremstyle{plain}
\newtheorem{theorem}{Theorem}
\newtheorem{lemma}{Lemma}
\newtheorem{corollary}{Corollary}

\theoremstyle{definition}
\newtheorem{remark}{Remark}

\title{On groups with unbounded Cayley graphs}

\author{Jakob Schneider}
\address{J.S., Institut f\"ur Geomtrie, TU Dresden, 01062 Dresden, Germany}
\email{jakob.schneider@tu-dresden.de}

\begin{document}

\begin{abstract}
	We show that every non-trivial compact connected group and every non-trivial general or special linear group over an infinite field admits a generating set such that the associated Cayley graph has infinite diameter.
\end{abstract}

\maketitle

A group $G$ is said to have the \emph{Bergman property} if for any symmetric generating subset $S\subseteq G$ containing the identity of $G$ there exists $k\in\nats$ such that $S^{\ast k}=G$, where $S^{\ast k}\defeq\set{s_1\cdots s_k}[s_1,\ldots,s_k\in S]$. The smallest such $k$ is called the \emph{width} of $S$ or the \emph{diameter} of the associated Cayley graph $\Cay(G,S)$. If there is no such $k$, $S$ is said to have infinite width.

Examples of such groups include finite groups, where the study of worst case diameters has a long history, but also infinite groups such as $\Sym(\nats)$ (a result due to Bergman \cite{bergman2006generating}). The first example of an infinite group with uniformly bounded width with respect to any generating set has been constructed by Shelah \cite[Theorem~2.1]{shelah1980problem}.

In \cite{dowerk2018bergman} Dowerk shows that unitary groups of \emph{infinite-dimensional} von Neumann algebras admit \emph{strong uncountable cofinality}. 
A group $G$ admits this property if for any exhausting chain $W_0\subseteq W_1\subseteq\cdots\subseteq G=\bigcup_{n=0}^\infty{W_n}$ of subsets there exist $n,k\in\nats$ such that $W_n^{\ast k}=G$ (see \cite[Definition~1.1]{rosendal2009topological}). It is apparent that, setting $W_n\defeq S^{\ast n}$, where $S$ is any symmetric generating set containing the identity, strong uncountable cofinality implies the Bergman property. Indeed, for uncountable groups it is actually equivalent to admitting both the Bergman property and \emph{uncountable cofinality} (see \cite[Proposition 2.2]{drostegoebel2005uncountable} or \cite{drosteholland2005generating}).
Regarding his result, Dowerk asked whether unitary groups of \emph{finite-dimensional} von Neumann algebras have the Bergman property.
In this note we answer this question in the negative by showing that any non-trivial compact connected group fails to have this property. We also show this for groups of type $\GL_n(K)$ or $\SL_n(K)\neq\trivgrp$ ($n\geq 1$), where $K$ is an infinite field.

Note that all of the above groups also fail to have uncountable cofinality. This is shown for $\SL_n(K)$ over an uncountable field in \cite[Proposition]{thomaszapletal2012steinhaus} and can easily be extended to all groups admitting a finite-dimensional representation with uncountable image. Such groups encompass compact groups by the Peter--Weyl theorem. This was pointed out by Cornulier in private communication.

Our main result is the following.

\begin{theorem}\label{thm:cpt_cn_grps} The following are true
\begin{enumerate}[(i)]
	\item Any non-trivial compact connected group does not admit the Bergman property.
	\item Any group of type $\GL_n(K)$ ($n\geq 1$) or $\SL_n(K)$ ($n\geq 2$), for an infinite field $K$, does not admit the Bergman property.
\end{enumerate}
\end{theorem}

The proofs use a few non-trivial facts from the theory of fields, for which we refer the reader to \cite{lang2002algebra}. For facts like the existence of a transcendence basis of $\reals$ over $\rats$ we need to assume the Axiom of Choice.

\section{Proof of Theorem~\ref{thm:cpt_cn_grps}}

This section is devoted to the proof of Theorem~\ref{thm:cpt_cn_grps}.
 
\begin{proof}[Proof of Theorem~\ref{thm:cpt_cn_grps}(\textnormal{i}):]
Let $\abs{\bullet}\colon\complex\to\reals_{\geq 0}$ be a \emph{norm} such that the following hold
\begin{enumerate}[(i)]
	\item $\abs{x}=0$ if and only if $x=0$ (identity of indiscernibles);
	\item $\abs{a+b}\leq\abs{a}+\abs{b}$ (subadditive);
	\item $\abs{ab}\leq\abs{a}\abs{b}$ (submultiplicative);
	\item $\abs{x^2}=\abs{x}^2$ (compatible with squaring);
	\item there exists $x\in[0,1]$ such that $\abs{x}>1$ (small elements with large norm).
\end{enumerate}
Examples of such norms are presented in Lemma~\ref{lem:cnst_nrms} below.

We now consider the matrix group $\SO(2,\reals)$. Set $C\defeq 4\abs{1/2}$. Then $C\geq 2\abs{1}=2$ by (ii) (as $\abs{1}=1$; see below). Define $S$ to be the set of elements of $\SO(2,\reals)$ with coefficients in $B\defeq\set{x\in\reals}[\abs{x}\leq C]$. Observe that $1\in S$, as (i) and (iv) imply that $\abs{1}=1$, and $S=S^{-1}$, as the inverse of $g\in \SO(2,\reals)$ is just $g^\top$. We claim that $S$ is a generating set and that $\SO(2,\reals)$ has infinite diameter with respect to $S$.
 
Assume that $g=(\begin{smallmatrix} a&b\\ -b&a \end{smallmatrix})\in \SO(2,\reals)$ corresponds to an element $z\in \U(1)\subseteq\complex$ via the isomorphism $(\begin{smallmatrix} a&b\\ -b&a \end{smallmatrix}) \mapsto a+bi$.
Take $k\in\nats$ large enough such that 
$$
\abs{z^{1/2^k}},\abs{z^{-1/2^k}}\leq 2.$$
This is possible by (iv).
Set $a'\defeq\Re(z^{1/2^k})$ and $b'\defeq\Im(z^{1/2^k})$. By property (iv) we have $\abs{i}=1$ as $\abs{i^4}=\abs{i}^4=\abs{1}=1$. From (ii), (iii) we derive that
$$
\abs{a'}=\lrabs{\frac{1}{2}(z^{1/2^k}+z^{-1/2^k})}\leq\lrabs{\frac{1}{2}}(\abs{z^{1/2^k}}+\abs{z^{-1/2^k}})\leq C.
$$
Similarly, using that $\abs{i}=1$, we obtain
$$
\abs{b'}=\lrabs{\frac{1}{2i}(z^{1/2^k}-z^{-1/2^k})}\leq\lrabs{\frac{1}{2}}(\abs{z^{1/2^k}}+\abs{z^{-1/2^k}})\leq C.
$$
Since $(a'+b'i)^{2^k}=z$, we conclude that $g\in S^{\ast2^k}$. Since $g$ was arbitrary, $S$ follows to be a generating set of $\SO(2,\reals)$.

It remains to show that $\SO(2,\reals)$ does not have finite diameter with respect to $S$. Indeed, for any $R\in\reals_{\geq 0}$ on finds $k$ large enough such that, if $x$ is an element as in (v), $\abs{x^{2^k}}=\abs{x}^{2^k}\geq R$. Now putting $a\defeq x^{2^n}$ and $b\defeq\sqrt{1-a^2}$, we obtain an element $g=a+bi\in\U(1)\cong\SO(2,\reals)$ which needs arbitrarily many factors in $S$ to be represented (taking $R$ large enough). Indeed, by induction, all coordinate entries of an element in $S^{\ast k}$ have norm at most $2^{k-1}C^k$ ($k\in\ints_+$).

Now let $H$ be a non-trivial compact connected Lie group represented as a matrix subgroup of $\O(n)$ for some $n\in\nats$. We observe that finitely many copies $H_i$ ($i=1,\ldots,m$) of $\SO(2,\reals)$ in $H$ generate $H$ as a group, e.g., see \cite[Theorem~2]{dalessandro2002uniform}. 
This allows us to extend our argument for $\SO(2,\reals)$ above to the group $H$ as follows. 
Assume that $H$ is generated by
$$
H_i= \lrset{g_i^{-1}\left(\begin{pmatrix} a & b\\-b&a \end{pmatrix}\oplus\id_{n-2}\right) g_i}[a,b\in\reals,\ a^2+b^2=1]
$$
for some $g_i\in\O(n)$ ($i=1,\ldots,m$). After conjugating by $g_1^{-1}$ we may assume that $g_1=1$.
As a generating set $S$ for $H$ it now suffices to take the set of elements in $H_i$ ($i=1,\ldots,m$) with coefficients $a,b\in B$, where $B$ is defined as above. Apparently, $S$ generates $H_i$ ($i=1,\ldots,m$) and hence $H$.
Let $D$ be the maximum of the norms of the matrix entries of the $g_i$ ($i=1,\ldots,m$; and so it is the maximum of the norms of the entries of the $g_i^{-1}=g_i^\top$). Then any coordinate entry of an element in $S\cap H_i$ has norm bounded by $c\defeq n^2D^2C$. Thus, by induction, for $g\in S^{\ast k}$, all coordinate entries of $g$ have norm at most $n^{k-1}c^k$ ($k\in\ints_+$). Therefore, $S$ cannot generate $H$ in finitely many steps, since, as above, the coordinate entries of $H_1\cong\SO(2,\reals)$ are unbounded with respect to the norm. 
 
If $H$ is now any non-trivial compact connected group, then by the Peter--Weyl theorem $H$ has a non-trivial finite-dimensional unitary representation, say $\pi \colon H\rightarrow \U(n) \subseteq\O(2n)$. Let $S$ be a generating set in $\O(2n)$ witnessing that $\pi(H)$ does not have the Bergman property. Then $\pi^{-1}(S)$ is a generating set witnessing that $H$ does not have the Bergman property. 
\end{proof}

It remains to construct norms satisfying (i)--(v). This is done in the subsequent lemma.

\begin{lemma}\label{lem:cnst_nrms}
	Let $T$ be a transcendence basis of $\reals$ over $\rats$. Consider $K\defeq\rats(T)$ and let $G$ denote the Galois group of the field extension $\complex/K$. The following norms $\abs{\bullet}\colon\complex\to\reals_{\geq 0}$ satisfy properties (i)--(v) from the proof Theorem~\ref{thm:cpt_cn_grps} (ii).
	\begin{enumerate}[(i)]
		\item Since, $K\subseteq\complex$ is algebraic, $G$ is a profinite group and each element in $\complex$ has a finite orbit under the action of $G$.
		For $x\in \complex$ we define the \emph{Galois radius} of $x$ to be 
		$\rho(x)\defeq\max_{\sigma\in G}\abs{x^\sigma}$. Then $\abs{\bullet}\defeq \rho$ defines a norm with the desired properties.
		\item Choose $t\in T$. Set $L\defeq \rats(T\setminus\set{t})$. Let $\nu\colon K\to\ints$ be the degree valuation corresponding to $t$ on $K$. Extend $\nu$ to a valuation $\omega\colon\complex\to\rats$ (by using Zorn's lemma). Setting $\abs{x}\defeq\exp(-\omega(x))$ gives a norm on $\complex$ satisfying the above properties.
	\end{enumerate}
	
\end{lemma}

\begin{proof}[Proof of Lemma~\ref{lem:cnst_nrms}]
	In both cases, we verify properties (i)--(v).
	
	\emph{(i):} Properties (i)--(iv) are immediate from the definition.
	We now show that also property (v) is fulfilled. Choose arbitrary real numbers $a\in(0,1)$ and $b>1$. Set $p(X)\defeq (X-a)(X-b)\in\reals[X]$. We claim that by density of $K\supseteq\rats$ in $\reals$ we can find an irreducible polynomial over $K$ with coefficients arbitrarily close to those of $p(X)$. Indeed, by Gauss' lemma, an irreducible polynomial over $\rats$ remains irreducible over $K$ and hence, it suffices to approximate by irreducible rational polynomials. Using Eisenstein's criterion, we can find an irreducible monic rational polynomial $q(X)$ which has coefficients arbitrarily close to the coefficients of $p(X)$. Indeed, Eisenstein's criterion implies that for $\alpha,\beta,\gamma\in\mathbb Z$, $\gamma>0$, the polynomial $q(X)=X^2+(\alpha/\gamma)X+(\beta/\gamma)$ is irreducible if $p^2$ does not divide $\beta$, $p$ divides $\alpha$ and $\beta$ and does not divide $\gamma$. Choosing $\gamma$ large enough and coprime to $p$, we can easily find $\alpha$ and $\beta$ with the desired properties such that $\alpha/\gamma$ is close to $-(a+b)$ and $\beta/\gamma$ is close to $ab$.
	By the implicit function theorem, the zeroes of $q$, say $x$ and $y$, are arbitrarily close to $a$ and $b$, respectively. Hence $\rho(x)$ is close to $b>1$, as desired.
	
	\emph{(ii):} Properties (i)--(iv) follow from the definition of a valuation. For property (v) observe that $\abs{t^{-1}}=\exp(\nu(t^{-1}))=e>1$. Also, for $y\in\rats$ we have $\abs{y}=\exp(\nu(y))=1$. Taking $y$ so that $x\defeq t^{-1}y\in[0,1]$, we obtain that for $\abs{x}=\abs{y}\abs{t^{-1}}=\abs{t^{-1}}=e>1$, as wished.
\end{proof}

\begin{remark}
	In the proof of Theorem~1(i) and Lemma~\ref{lem:cnst_nrms} the field $\reals$ can be replaced by a Euclidean field $R$ and $\complex$ by $R[i]$. Then, if $T=\emptyset$ in Lemma~\ref{lem:cnst_nrms}(ii) we need to take a $p$-adic valuation, instead of the degree valuation, on $K=\rats$ and extend it. Then, for $y$, in the above argument, we have to take a suitable element $r/s\in\rats$ such that $p\nmid r,s$.
\end{remark}

\begin{proof}[Proof of Theorem~\ref{thm:cpt_cn_grps}\textnormal{(ii)}]
	At first we consider the case $G=\SL_n(K)$. We distinguish two cases.
	
	\emph{Case~1:} Assume first that a transcendence basis $T$ of $K$ over its prime field $k$ is not empty. Define for $\lambda\in K^\times$, $\mu\in K$ the matrices
	$$
	D(\lambda)\defeq\begin{pmatrix}
	\lambda^{-1} & 0\\
	0 & \lambda
	\end{pmatrix}
	\text{ and }
	E_{12}(\mu)\defeq\begin{pmatrix}
	1 & \mu\\
	0 & 1
	\end{pmatrix}.	
	$$
	Note that 
	\begin{equation}\label{eq:el_mats_id}
	E_{12}(\mu)^{D(\lambda_1)\cdots D(\lambda_n)}=E_{12}(\mu)^{D(\lambda_1\cdots\lambda_n)}=E_{12}((\lambda_1\cdots\lambda_n)^2\mu)
	\end{equation}
	for $\lambda_1,\ldots,\lambda_n\in K^\times$, $\mu\in K$. Choose $t\in T$. For $L\defeq k(T\setminus\set{t})$  we have that $K$ is algebraic over $L(t)$.
	Take the normed degree valuation on $\nu\colon L(t)\to\ints$ and extend it to a valuation $\omega\colon K\to\rats$. Define $B\defeq\set{b\in K}[\abs{\omega(b)}\leq 1]$.
	Set 
	$$
	S\defeq\set{g\in\SL_n(K)}[\text{all entries of }g\text{ are in }B].
	$$
	Then $S^{\ast k}\neq\SL_n(K)$, since one sees easily by induction that for $g\in S^{\ast k}$ $\abs{\omega(g_{ij})}\leq k$ (from the strong triangle inequality). We show that $S$ generates all elementary matrices, and hence $\SL_n(K)$.
	Indeed, we may assume that $n=2$, via the embeddings $\SL_2(K)\hookrightarrow\SL_n(K)$. We show that $S$ generates any matrix $M\defeq E_{12}(\alpha)$ for $\alpha\in K$. Indeed, if $\alpha=0$, then $M\in S$. In the opposite case $v\defeq\omega(\alpha)\in\rats$ and we find an integer $n\in\ints$ such that $\abs{v-2n}\leq 1$.
	Then choose $\lambda\in K$ such that $\omega(\lambda)=1$ (which exists even in $L(t)\subseteq K$, since $\nu$ was normed) and set $\mu\defeq\alpha\lambda^{-2n}$. By construction, we have $\abs{\omega(\mu)}\leq 1$, so that $D(\lambda),E_{12}(\mu)\in S$ and $M=E_{12}(\alpha)=E_{12}(\mu)^{D(\lambda)^{2n}}\in\gensubgrp{S}$ by Equation~\eqref{eq:el_mats_id}. This completes the proof of Case~1.
	
	\emph{Case~2:} In this case $K$ is an algebraic extension of its prime field $k=\rats$ or $k=\finfield_p$. In the first case, $K$ embeds into $\complex$ and we can define 
	$$
	S\defeq\set{g\in\SL_n(K)}[\text{all entries of }g\text{ are in }B],
	$$
	where $B$ is the unit ball of $\complex$ intersected with $K\subseteq\complex$.
	
	Hence we assume that $k=\finfield_p$ and $K$ is an infinite algebraic extension of $k$.
	In this case, we construct a set $B\subseteq K$ with the following properties
	\begin{enumerate}[(i)]
		\item $\gensubgrp{B}_{+}=K$, i.e., $B$ generates $K$ as an abelian group;
		\item $P(B)\defeq\set{p(b_1,\ldots,b_m)}[b_1,\ldots,b_m\in B, p\in P]\neq K$ for each finite set $P\subseteq\ints[X_1,\ldots,X_m]$ of polynomials over $\ints$ and $m\in\nats$.	
	\end{enumerate}
	Indeed, if we have such a set $B$, we can use the same definition for $S$ as above. This is due to the fact that $E_{12}(\mu)E_{12}(\lambda)=E_{12}(\mu+\lambda)$ and the elementary matrices generate $\SL_n(K)$. On the other hand $S^{\ast k}\neq\SL_n(K)$ by condition (ii), as the matrix entries are bounded degree polynomials over $\ints$ in the entries of the matrices.
	
	The set $B$ is now inductively constructed in Lemma~\ref{lem:comb_finflds} and Corollary~\ref{cor:cnst_B} below.
	
	In the case $G=\GL_n(K)$ we add matrices $\diag(\lambda,1,\ldots,1)$ to the generating set $S$ with $\abs{\omega(\lambda)}\leq 1$.
\end{proof}

\begin{lemma}\label{lem:comb_finflds}
	Fix a set $P\subseteq\finfield_p[X_1,\ldots,X_m]$ of non-constant polynomials of total degree at most $n$, i.e., especially $P$ is finite.
	Consider the inclusion of finite fields $\finfield_{p^e}\subseteq\finfield_{p^{ef}}$ for $e,f\in\ints_+$ and let $E\subseteq\finfield_{p^e}$ be a subset of cardinality $e$. Define $P_E\subseteq\finfield_{p^e}[X_1,\ldots,X_m]$ as the set of non-constant polynomials which arise from the polynomials from $P$ by substituting a subset of the variables $X_1,\ldots,X_m$ by elements from $E$. Then for $e$ sufficiently large there exists a set $F\subseteq\finfield_{p^{ef}}$ such that $\gensubsp{F}_{\finfield_p}$ is a complement to $\finfield_{p^e}$ in $\finfield_{p^{ef}}$ as $\finfield_p$-vector spaces and $r(f_1,\ldots,f_m)\not\in\finfield_{p^e}$ for all $f_1,\ldots,f_m\in F$ and $r\in P_E$.
\end{lemma}

\begin{proof}
	At first note that the set $C$ of $e(f-1)$-tuples with entries in $\finfield_{p^{ef}}$ which span an $\finfield_p$-complement of $\finfield_{p^e}$ in $\finfield_{p^{ef}}$ is of cardinality
	$(p^{ef}-p^e)\cdots(p^{ef}-p^{ef-1})$, so its portion in the set $T$ of all $e(f-1)$-tuples with entries in $\finfield_{p^{ef}}$ is equal to
	$$
	\card{C}/p^{e^2f(f-1)}=(1-p^{-e(f-1)})\cdots(1-p^{-1})
	$$
	but for $c\defeq 2\log(2)$ we have $e^{-cx}\leq 1-x$ for $0\leq x\leq 1/2$, so that the above expression is bounded from below by
	$$
	e^{-c\sum_{i=1}^{e(f-1)}{p^{-i}}}\geq d\defeq e^{-\frac{c}{p-1}}\in (0,1).
	$$
	
	Now let us estimate how many tuples $t=(t_1,\ldots,t_{e(f-1)})\in T$ have the property that $r(s_1,\ldots,s_m)\not\in\finfield_{p^e}$ for all $s_1,\ldots,s_m\in\overline{t}\defeq\set{t_1,\ldots,t_{e(f-1)}}$ and $r\in P_E$. By the Schwartz-Zippel lemma, for all $x\in\finfield_{p^e}$ we have
	$$
	\Prob_{s_1,\ldots,s_n\in\finfield_{p^{ef}}}[r(s_1,\ldots,s_m)=x]\leq n/p^{ef},
	$$
	so that
	$$
	\Prob_{s_1,\ldots,s_n\in\finfield_{p^{ef}}}[r(s_1,\ldots,s_m)\in\finfield_{p^e}]\leq n/p^{e(f-1)},
	$$
	and hence
	$$
	\Prob_{t\in T}[\exists r\in P_E,s_1,\ldots,s_m\in\overline{t}:r(s_1,\ldots,s_m)\in\finfield_{p^e}]\leq \card{P_E}(e(f-1))^m n/p^{e(f-1)}.
	$$
	Putting both estimates together, we obtain that the set of $t\in T$ such that $F\defeq\overline{t}$ satisfies the hypothesis of the lemma has portion bounded from below by $d-\card{P_E}(e(f-1))^m n/p^{e(f-1)}\geq d-(e+1)^m\card{P}(e(f-1))^m n/p^{e(f-1)}$ in $T$. But this term is clearly positive for all $f>1$ when $e$ is large enough.
\end{proof}

\begin{corollary}\label{cor:cnst_B}
	Let $K\subseteq\overline{\finfield}_p$ be infinite. There exists a set $B\subseteq K$ satisfying (i) and (ii) in the proof of Theorem~\ref{thm:cpt_cn_grps}(ii).
\end{corollary}

\begin{proof}
	Set $P_i\defeq\finfield_p[X_1,\ldots,X_i]_{\deg\leq i}$ for $i\in\nats$, so that the $P_i$ exhaust the polynomial ring over $\finfield_p$ in the countably many variables $X_i$ ($i\in\ints_+$). 
	Apply Lemma~\ref{lem:comb_finflds} to $P=P_0$, and choose $b_0\defeq e$ large enough and an $\finfield_p$-basis $B_0\defeq E$ of $\finfield_{p^{b_0}}$ such that for all $f>1$ and an appropriate extension $F\subseteq\finfield_{p^{b_0f}}$ to a basis $B_1\defeq B_0\cup F$ of this field, we have $P_{0,B_0}(F)\not\in\finfield_{p^{b_0}}$. Then, again using Lemma~\ref{lem:comb_finflds}, choose $f>1$ large enough, set $b_1\defeq b_0f$ and $B_1= B_0\cup F$ such that the same as above holds, replacing $P_0$ by $P_1$ and $b_0$ by $b_1$.
	Proceed by induction to get $B\defeq\bigcup_{i=0}^\infty{B_i}$.
	Apparently, $\gensubsp{B}_+=K$. Assume now that for $P\subseteq\finfield_p[X_1,X_2,\ldots]$ finite, we have $P(B)=K$.
	Then choose $m$ large enough such that $P\subseteq P_m$.
	Now note that
	$$
	P_m(B)=P_m(B_m)\cup\bigcup_{i=m}^\infty{P_{m,B_i}(B_{i+1}\setminus B_i)}.
	$$
	But $P_m(B_m)\subseteq\finfield_{p^{b_m}}$ and $P_{m,B_i}(B_{i+1}\setminus B_i)\subseteq P_{i,B_i}(B_{i+1}\setminus B_i)\subseteq\finfield_{p^{b_{i+1}}}\setminus\finfield_{p^{b_i}}$ for $i\geq m$ by construction.
	
	Hence for $i\geq m$ we have $P_m(B)\cap\finfield_{p^{b_i}}=P_m(B_i)$. But the size of this set is bounded by $\card{P_m}\card{B_i}^m=\card{P_m}b_i^m$, whereas $\card{\finfield_{p^{b_i}}}=p^{b_i}$, which is eventually larger than the first. 
	
	This shows that $P(B)\subseteq P_m(B)\neq K$, as desired.
\end{proof}

We end up with a question: Does there exist a countably infinite group admitting the Bergman property?

\section*{Acknowledgments}
The author wants to thank Philip Dowerk and Andreas Thom for interesting discussions and Yves Cornulier for a comment concerning a first version of this article.
This research was supported by ERC Consolidator Grant No.\ 681207.

\end{document}